\documentclass[11pt]{article}%
\usepackage{amsfonts}
\usepackage{amsmath}
\usepackage{amssymb}
\usepackage{graphicx}%
\newtheorem{theorem}{Theorem}

\newtheorem{definition}[theorem]{Definition}
\newtheorem{example}[theorem]{Example}

\newtheorem{proposition}[theorem]{Proposition}
\newtheorem{remark}[theorem]{Remark}

\newenvironment{proof}[1][Proof]{\noindent\textbf{#1.} }{\ \rule{0.5em}{0.5em}}
\begin{document}

\title{Complete systems of recursive integrals and Taylor series for solutions of
Sturm-Liouville equations}
\author{Vladislav V. Kravchenko$^{1}$, Samy Morelos$^{1}$, S\'{e}bastien
Tremblay$^{2}$
\and \medskip\\$^{1}$Departamento de Matem\'{a}ticas, CINVESTAV del IPN, Unidad\\Quer\'{e}taro, Libramiento Norponiente No.~2000 C.P. 76230 Fracc.\\Real de Juriquilla, Quer\'{e}taro, Mexico\\$^{2}$D\'{e}partement de math\'{e}matiques et d'informatique, Universit\'{e} du\\Qu\'{e}bec, Trois-Rivi\`{e}res, Qu\'{e}bec, G9A 5H7, Canada}
\maketitle

\begin{abstract}
Consider an arbitrary complex-valued, twice continuously differentiable,
nonvanishing function $\varphi$ defined on a finite segment $[a,b]\subset
\mathbb{R}$. Let us introduce an infinite system of functions constructed in
the following way. Each subsequent function is a primitive of the preceding
one multiplied or divided by $\varphi$ alternately. The obtained system of
functions is a generalization of the system of powers $\left\{  (x-x_{0}%
)^{k}\right\}  _{k=0}^{\infty}$. We study its completeness as well as the
completeness of its subsets in different functional spaces. This system of
recursive integrals results to be closely related to so-called $L$-bases
arising in the theory of transmutation operators for linear ordinary
differential equations. 

Besides the results on the completeness of the system of recursive integrals
we show a deep analogy between the expansions in terms of the recursive
integrals and Taylor expansions. We prove a generalization of the Taylor
theorem with the Lagrange form of the remainder term and find an explicit
formula for transforming a generalized Taylor expansion of a function in terms
of the recursive integrals into a usual Taylor expansion. As a direct
corollary of the formula we obtain the following new result concerning
solutions of the Sturm-Liouville equation. Given a regular nonvanishing
complex valued solution $y_{0}$ of the equation $y^{\prime\prime}+q(x)y=0$,
$x\in(a,b)$, assume that it is $n$ times differentiable at a point $x_{0}%
\in\lbrack a,b]$. We present explicit formulas for calculating the first $n$
derivatives at $x_{0}$ for any solution of the equation $u^{\prime\prime
}+q(x)u=\lambda u$. That is, an explicit map transforming the Taylor expansion
of $y_{0}$ into the Taylor expansion of $u$ is constructed. \\



\end{abstract}

\section{Introduction}

In the recent work \cite{KrCV08} by means of pseudoanalytic function theory
\cite{Berskniga}, \cite{APFT} a representation for solutions of
Sturm-Liouville equations in terms of spectral parameter power series was
obtained. For a proof not requiring pseudoanalytic functions we refer to
\cite{APFT} and \cite{KrPorter2010}. The representation turned to be an
appropriate tool for solving different Sturm-Liouville and related problems
\cite{CamposKr}, \cite{CKKO}, \cite{CKOR}, \cite{KhmRosu}, \cite{KhmRosuHill},
\cite{KhmRosuDirac}, \cite{KhmRosuGonzalez}, \cite{KhmTorch},
\cite{KrPorter2010}, \cite{KPConformal} and \cite{KrVelasco}. As is well known
(see, e.g., \cite{Levitan}) under certain regularity conditions a solution to
the initial value problem for the Sturm-Liouville equation is an analytic
function of the spectral parameter and hence admits a Taylor series expansion
in powers of the spectral parameter. In fact in \cite{KrCV08} the recursive
formulas for calculating the Taylor coefficients in that expansion were
proposed. The Taylor coefficients are naturally functions of the independent
variable.\ Considered as an infinite system of functions they form a so-called
$L$-basis related to a linear differential operator $L$. The theory of
$L$-bases and $L$-analytic functions was developed in \cite{Fage}
(unfortunately that interesting and important book has not been translated
into English) in relation with the concept of the operators of transmutation
(see, e.g., \cite{Gilbert}, \cite{Carroll}, \cite{Trimeche} and the recent
review \cite{Sitnik}), also called in Russian bibliography operators of
transformation (see, e.g., \cite{LevitanInverse}, \cite{Marchenko}).

In \cite{Fage} it was shown that to every (regular) linear ordinary
differential operator one can associate a linear space spanned on an
$L$-basis, and the equivalence of those linear spaces corresponding to
operators of the same order was proved. Then the operator of transmutation can
be regarded as an operation which transforms functions from one such linear
space corresponding to a certain operator $L$ to functions from another linear
space corresponding to another operator $M$, and the transformation consists
in substituting the $L$-basis with the $M$-basis keeping the same coefficients
in the expansion. In the recent work \cite{KrCMA2011} it was shown that the
situation with the linear space generated by the $L$-basis is in a sense
simpler and more natural. Namely, the $L$-basis corresponding to a regular
Sturm-Liouville operator on a finite interval $(a,b)$ is complete in
$L_{2}(a,b)$. Moreover, in the present work we prove that in the space of
piecewise continuously differentiable functions on $[a,b]$ the $L$-basis is
complete with respect to the maximum norm (see theorem \ref{ThComplMaxNorm}).
We show a deep analogy of the $L$-basis with the system $\left\{
(x-x_{0})^{n}\right\}  _{n=0}^{\infty}$ of powers of the independent variable.
The formulas for calculating coefficients in the expansion of a given function
in terms of those generalized powers resulted to be direct generalizations of
the formulas for Taylor coefficients. We prove a generalization of the Taylor
theorem with the Lagrange form of the remainder term.

A natural question then is about a relation between the generalized Taylor
coefficients of a function and its ordinary Taylor coefficients. This question
has not been studied previously. We find this relation and use it to obtain
the main result of this paper, the representation for ordinary Taylor
coefficients of any solution of the Sturm-Liouville equation $u^{\prime\prime
}+q(x)u=\lambda u$ in terms of the Taylor coefficients of a particular
solution $y_{0}$ of the equation $y^{\prime\prime}+q(x)y=0$. More precisely we
obtain the Taylor coefficients of the quotient $u/y_{0}$. It is interesting to
notice that the form of the map transforming one set of Taylor coefficients
into the other is independent of the particular form of $y_{0}$. The paper
ends with examples of application of the presented result.

\section{Definition and example}

Let $f\in C^{2}(a,b)\cap C^{1}[a,b]$ be a complex valued function and
$f(x)\neq0$ for any $x\in\lbrack a,b]$. The interval $(a,b)$ is supposed to be
finite. Let us consider the following auxiliary functions%
\begin{equation}
\widetilde{X}^{(0)}(x)\equiv X^{(0)}(x)\equiv1, \label{X1}%
\end{equation}%
\begin{equation}
\widetilde{X}^{(n)}(x)=n%
%TCIMACRO{\dint \limits_{x_{0}}^{x}}%
%BeginExpansion
{\displaystyle\int\limits_{x_{0}}^{x}}
%EndExpansion
\widetilde{X}^{(n-1)}(s)\left(  f^{2}(s)\right)  ^{(-1)^{n-1}}\,\mathrm{d}s,
\label{X2}%
\end{equation}%
\begin{equation}
X^{(n)}(x)=n%
%TCIMACRO{\dint \limits_{x_{0}}^{x}}%
%BeginExpansion
{\displaystyle\int\limits_{x_{0}}^{x}}
%EndExpansion
X^{(n-1)}(s)\left(  f^{2}(s)\right)  ^{(-1)^{n}}\,\mathrm{d}s, \label{X3}%
\end{equation}
where $x_{0}$ is an arbitrary fixed point in $[a,b]$. We introduce the
infinite systems of functions%
\begin{equation}
\left\{  f_{n}\right\}  _{n=1}^{\infty}\qquad\text{and}\qquad\left\{
g_{n}\right\}  _{n=1}^{\infty} \label{fg}%
\end{equation}
defined by the relations%
\begin{equation}
f_{n}=f\widetilde{X}^{(2(n-1))}\qquad\text{and}\qquad g_{n}=fX^{(2n-1)}.
\label{fngn}%
\end{equation}
In \cite{KrCMA2011} it was shown that both systems of functions in (\ref{fg})
are complete in $L_{2}(a,b)$ when the point $x_{0}$ coincides with one of the
end points of the interval, and if $x_{0}$ is an interior point of the
interval, then the completeness of the union of two systems $\left\{
f_{n}\right\}  _{n=1}^{\infty}\cup\left\{  g_{n}\right\}  _{n=1}^{\infty}$ can
be guaranteed.

The question about the completeness of the systems (\ref{fg}) is natural due
to the following observation.

\begin{example}
\label{ExamplePoly}Let $f\equiv1$, $a=0$, $b=1$. Then it is easy to see that
choosing $x_{0}=0$ we have
\begin{equation}
f_{1}(x)=1,\quad f_{2}(x)=x^{2},\quad f_{3}(x)=x^{4},\ldots\label{polyeven}%
\end{equation}
and
\begin{equation}
g_{1}(x)=x,\quad g_{2}(x)=x^{3},\ldots. \label{polyodd}%
\end{equation}
As is well known due to the M\"{u}ntz theorem (see, e.g., \cite[p. 270]%
{Davis}) both systems of polynomials are complete in $L_{2}(0,1)$. The systems
(\ref{fg}) represent a direct generalization of the systems of polynomials
(\ref{polyeven}) and (\ref{polyodd}) if instead of $f\equiv1$ an arbitrary
sufficiently smooth and nonvanishing function is chosen.
\end{example}

The completeness of the system (\ref{polyeven}) can be regarded as a corollary
of the completeness of the system of eigenfunctions of the Sturm-Liouville
problem%
\begin{equation}
u^{\prime\prime}+\lambda u=0,\qquad0<x<1, \label{SL0}%
\end{equation}%
\begin{equation}
u^{\prime}(0)=u^{\prime}(1)=0. \label{condDer}%
\end{equation}
Indeed, the eigenfunctions of this regular Sturm-Liouville problem have the
form $u_{n}(x)=\cos n\pi x,$ $n=0,1,2,\ldots$, and as each of them admits a
uniformly convergent Taylor expansion in even powers of $x$, the system
(\ref{polyeven}) is also complete.

Analogously, considering (\ref{SL0}) with the boundary conditions
$u(0)=u(1)=0$ we arrive at the complete system of eigenfunctions
$v_{n}(x)=\sin n\pi x,$ $n=1,2,\ldots.$ Each of them admits a uniformly
convergent Taylor expansion in odd powers of $x$ and hence the system
(\ref{polyodd}) is complete in $L_{2}(0,1)$ as well.

\section{Completeness of the systems of recursive integrals}

We will need the following theorem from \cite{KrCV08} (for additional details
see \cite{APFT} and \cite{KrPorter2010}) establishing the relation of the
systems of functions (\ref{fg}) to Sturm-Liouville equations.

\begin{theorem}
\label{ThGenSolSturmLiouville}\cite{KrCV08} Let $q$ be a continuous complex
valued function of an independent real variable $x\in\lbrack a,b],$ $\lambda$
be an arbitrary complex number. Suppose there exists a solution $f$ of the
equation
\begin{equation}
f^{\prime\prime}+qf=0 \label{SLhom}%
\end{equation}
on $(a,b)$ such that $f\in C^{2}(a,b)$ together with $1/f$ are bounded on
$[a,b]$. Then the general solution of the equation $u^{\prime\prime
}+qu=\lambda u$ on $(a,b)$ has the form%
\[
u=c_{1}u_{1}+c_{2}u_{2}%
\]
where $c_{1}$ and $c_{2}$ are arbitrary complex constants,
\begin{equation}
u_{1}=f%
%TCIMACRO{\dsum \limits_{k=0}^{\infty}}%
%BeginExpansion
{\displaystyle\sum\limits_{k=0}^{\infty}}
%EndExpansion
\frac{\lambda^{k}}{(2k)!}\widetilde{X}^{(2k)}\quad\text{and}\quad u_{2}=f%
%TCIMACRO{\dsum \limits_{k=0}^{\infty}}%
%BeginExpansion
{\displaystyle\sum\limits_{k=0}^{\infty}}
%EndExpansion
\frac{\lambda^{k}}{(2k+1)!}X^{(2k+1)} \label{u1u2}%
\end{equation}
with $\widetilde{X}^{(n)}$ and $X^{(n)}$ being defined by (\ref{X1}%
)-(\ref{X3}) and both series converge uniformly on $[a,b]$.
\end{theorem}

\begin{remark}
\label{RemInitialValues}It is easy to see that by definition the solutions
$u_{1}$ and $u_{2}$ satisfy the following initial conditions
\begin{equation}
u_{1}(x_{0})=f(x_{0}),\qquad u_{1}^{\prime}(x_{0})=f^{\prime}(x_{0}),
\label{initial1}%
\end{equation}%
\begin{equation}
u_{2}(x_{0})=0,\qquad u_{2}^{\prime}(x_{0})=1/f(x_{0}). \label{initial2}%
\end{equation}

\end{remark}

\begin{remark}
It is worth mentioning that in the regular case the existence and construction
of the required $f$ presents no difficulty. Let $q$ be real valued and
continuous on $[a,b]$. Then (\ref{SLhom}) possesses two linearly independent
regular solutions\/ $v_{1}$ and $v_{2}$ whose zeros alternate. Thus one may
choose $f=v_{1}+iv_{2}$.
\end{remark}

\begin{theorem}
\cite{KrCMA2011} Let $(a,b)$ be a finite interval and $f\in C^{2}(a,b)\cap
C^{1}[a,b]$ be a complex valued function such that $f(x)\neq0$ for any
$x\in\lbrack a,b]$. Then both systems of functions (\ref{fg}) defined by the
relations (\ref{fngn}) and (\ref{X1})-(\ref{X3}) with $x_{0}=a$ are complete
in $L_{2}(a,b)$.
\end{theorem}

\begin{proof}
Under the conditions of the theorem, on the interval $(a,b)$ the function $f$
is a regular solution of the equation (\ref{SLhom})\ \ with $q:=-f^{\prime
\prime}/f$ being a continuous complex valued function. Let us consider the
equation
\begin{equation}
u^{\prime\prime}+qu=\lambda u \label{SL}%
\end{equation}
with the boundary conditions%
\begin{equation}
u(a)=u(b)=0. \label{problem1}%
\end{equation}
It is known (see \cite[p. 36]{Marchenko}) that the system of all
eigenfunctions and generalized eigenfunctions of this problem is complete in
$L_{2}(a,b)$. Due to theorem \ref{ThGenSolSturmLiouville} and remark
\ref{RemInitialValues} if $\lambda_{n}$ is an eigenvalue of the problem
(\ref{SL}), (\ref{problem1}) then the corresponding eigenfunction up to a
constant factor must coincide with $u_{2}$ from (\ref{u1u2}) where
$\lambda=\lambda_{n}$, i.e., it admits a uniformly convergent series expansion
in terms of the system of functions $g_{n}$ from (\ref{fngn}). Moreover, if
the multiplicity of $\lambda_{n}$ is greater than $1$ then the corresponding
generalized eigenfunctions are obtained differentiating the eigenfunction with
respect to the spectral parameter $\lambda$ (see \cite[p. 27]{Marchenko}). The
representation (\ref{u1u2}) of $u_{2}$ shows us that the result of this
operation will be again a series in terms of the functions $g_{n}$. Thus, the
eigenfunctions and the generalized eigenfunctions of the problem (\ref{SL}),
(\ref{problem1}) can be represented as uniformly convergent series in terms of
the system of functions $\left\{  g_{n}\right\}  _{n=1}^{\infty}$. Then due to
the Lauricella theorem about the transitivity of the property \ of
completeness (see, e.g., \cite[p. 264]{Davis}) we obtain that the system
$\left\{  g_{n}\right\}  _{n=1}^{\infty}$ is complete in $L_{2}(a,b)$.

In a similar way the completeness of $\left\{  f_{n}\right\}  _{n=1}^{\infty}$
is proved by considering the Sturm-Liouville problem for equation (\ref{SL})
with the boundary conditions%
\[
f^{\prime}(a)u(a)-f(a)u^{\prime}(a)=u(b)=0.
\]
All eigenfunctions of this problem coincide with $u_{1}$ from (\ref{u1u2})
where $\lambda=\lambda_{n}$ (this is due to remark \ref{RemInitialValues}) and
generalized eigenfunctions are obtained from the eigenfunctions by
differentiation with respect to $\lambda$. Thus, by analogy with the
previously considered case all the eigenfunctions and generalized
eigenfunctions of this Sturm-Liouville problem are represented as uniformly
convergent series in terms of $f_{n}$, and by the Lauricella theorem $\left\{
f_{n}\right\}  _{n=1}^{\infty}$ is complete in $L_{2}(a,b)$.
\end{proof}

In this theorem we assumed that the point $x_{0}$ coincided with one of the
end points of the interval $(a,b)$. When $x_{0}$ is an interior point of the
interval in general the systems $\left\{  f_{n}\right\}  _{n=1}^{\infty}$ and
$\left\{  g_{n}\right\}  _{n=1}^{\infty}$ separately are not complete. It is
easy to see on the considered above example \ref{ExamplePoly} that if $a=-1$
and all other values remain unchanged then the system (\ref{polyeven}) is not
complete anymore in $L_{2}(a,b)$ because (\ref{polyeven}) contains even
functions only. Nevertheless considering the combined system of functions
$\left\{  f_{n}\right\}  _{n=1}^{\infty}\cup\left\{  g_{n}\right\}
_{n=1}^{\infty}$ we obtain a complete system. The following theorem
establishes that this remains true in a much more general situation.

\begin{theorem}
\cite{KrCMA2011}\label{ThComplAnyX0} Let $(a,b)$ be a finite interval and
$f\in C^{2}(a,b)\cap C^{1}[a,b]$ be a complex valued function such that
$f(x)\neq0$ for any $x\in\lbrack a,b]$. Then the system of functions $\left\{
f_{n}\right\}  _{n=1}^{\infty}\cup\left\{  g_{n}\right\}  _{n=1}^{\infty}$
defined by the relations (\ref{fngn}) and (\ref{X1})-(\ref{X3}) with $x_{0}$
being an arbitrary point of the interval $[a,b]$ are complete in $L_{2}(a,b)$.
\end{theorem}

\begin{proof}
As in the proof of the previous theorem let us consider equation (\ref{SL})
where $q:=-f^{\prime\prime}/f$, for example, with the boundary conditions
(\ref{problem1}). Due to theorem \ref{ThGenSolSturmLiouville} any
eigenfunction as well as any generalized eigenfunction of this problem can be
represented as a uniformly convergent series in terms of the functions $f_{n}$
and $g_{n}$. From the completeness of the system of eigenfunctions and
generalized eigenfunctions in $L_{2}(a,b)$ and by means of the Lauricella
theorem we obtain the result.
\end{proof}

\begin{remark}
If instead of the \textquotedblleft seed\textquotedblright\ function $f$ one
considers the function $1/f$ then, as it is easy to see the auxiliary
functions $\widetilde{X}^{(n)}$ and $X^{(n)}$ change their respective roles,
and in the same way as was done above one can prove the completeness of the
systems%
\[
\Big\{\frac{1}{f}X^{(2(n-1))}\Big\}_{n=1}^{\infty}\qquad\text{and}%
\qquad\Big\{\frac{1}{f}\widetilde{X}^{(2n-1)}\Big\}_{n=1}^{\infty}%
\]
in $L_{2}(a,b)$ when $x_{0}$ coincides with $a$ or $b$ and the completeness of
the union of these two systems in $L_{2}(a,b)$ when $x_{0}$ is an arbitrary
point of the interval $[a,b]$.
\end{remark}

\begin{remark}
Under the considered conditions on the function $f$ it is easy to prove the
completeness of the systems $\big\{  X^{(2(n-1))}\big\}  _{n=1}^{\infty}$,
$\big\{  X^{(2n-1)}\big\}  _{n=1}^{\infty}$, $\big\{  \widetilde{X}%
^{(2n-1)}\big\}  _{n=1}^{\infty}$, $\big\{  \widetilde{X}^{(2(n-1))}%
\big\}  _{n=1}^{\infty}$ when $x_{0}$ coincides with $a$ or $b$ and the
completeness of the systems $\big\{  \widetilde{X}^{(2(n-1))}\big\}
_{n=1}^{\infty}\cup\big\{  X^{(2n-1)}\big\}  _{n=1}^{\infty}$ and $\big\{
\widetilde{X}^{(2n-1)}\big\}  _{n=1}^{\infty}\cup\left\{  X^{(2(n-1))}%
\right\}  _{n=1}^{\infty}$ when $x_{0}$ is an arbitrary point of the interval
$[a,b]$. This is based on the previous theorems and on the observation that if
$\left\{  h_{n}\right\}  _{n=1}^{\infty}$ is a complete system in $L_{2}(a,b)$
then taking a continuous in $[a,b]$ weight function $p$, such that
$1/\left\vert p\right\vert $ is separated from zero one has that $\left\{
p\,h_{n}\right\}  _{n=1}^{\infty}$ is complete in $L_{2}(a,b)$ as well. For
the proof see \cite[Sect. 4.7]{Collatz}.
\end{remark}

Let us introduce the system of functions $\left\{  \varphi_{k}\right\}
_{k=0}^{\infty}$ defined as follows%

\begin{equation}
\varphi_{k}(x)=\left\{
\begin{tabular}
[c]{ll}%
$f(x)X^{(k)}(x)$, & $k$ \text{odd,}\\
$f(x)\widetilde{X}^{(k)}(x)$, & $k$ \text{even,}%
\end{tabular}
\ \ \ \ \ \ \ \ \ \right.  \ \label{phik}%
\end{equation}
where the definition of $X^{(k)}$ and $\widetilde{X}^{(k)}$ is given by
(\ref{X1})-(\ref{X3}) with $x_{0}$ being an arbitrary point of the interval
$[a,b]$. We are interested in the completeness of this system in the space of
piecewise differentiable functions with respect to the maximum norm and in the
corresponding series expansions.

The system of eigenfunctions of the Sturm-Liouville problem under certain
regularity conditions is complete not only in the sense of the $L_{2}$-norm
but also in the uniform convergence topology. The approach used in the proof
of the previous two theorems and based on the completeness results for the
system of eigenfunctions and generalized eigenfunctions of the Sturm-Liouville
problem as well as on the Lauricella theorem about the transitivity of the
property \ of completeness gives us the following result.

\begin{theorem}
\label{ThComplMaxNorm}Let $f$ satisfy the conditions of theorem
\ref{ThComplAnyX0} and $\left\{  \varphi_{k}\right\}  _{k=0}^{\infty}$ be the
system of functions defined by (\ref{phik}) with $x_{0}$ being an arbitrary
point of the interval $[a,b]$. Then for any complex valued piecewise
continuously differentiable function $h$ defined on $[a,b]$ and for any
$\varepsilon>0$ there exists such $N\in\mathbb{N}$ and such complex numbers
$\alpha_{k}$, $k=0,1,\ldots N$ that $\max_{x\in\lbrack a,b]}\left\vert
h(x)-\sum\limits_{k=0}^{N}\alpha_{k}\varphi_{k}\right\vert <\varepsilon$.
\end{theorem}

\begin{proof}
Take a piecewise continuously differentiable function $h$. It satisfies
certain boundary conditions of the form
\begin{equation}
c_{1}h(a)+c_{2}h^{\prime}(a)=0\quad\text{and\quad}c_{3}h(b)+c_{4}h^{\prime
}(b)=0 \label{bch}%
\end{equation}
where $c_{1}^{2}+c_{2}^{2}\neq0$ and $c_{3}^{2}+c_{4}^{2}\neq0$. Let $\left\{
u_{k}\right\}  _{k=0}^{\infty}$ be a system of all eigenfunctions and
generalized eigenfunctions of the Sturm-Liouville problem for equation
(\ref{SL}) with the non-degenerate boundary conditions (\ref{bch}). This
system is complete in the linear space of piecewise continuously
differentiable functions with respect to the maximum norm (see \cite[Chapter
1]{Marchenko}). As was previously shown every $u_{k}$ admits a uniformly
convergent series expansion in terms of the functions $\varphi_{k}$.
Thus\ again, the result is a corollary of the Lauricella theorem.
\end{proof}

\begin{remark}
In what follows instead of the system of functions $\left\{  \varphi
_{k}\right\}  _{k=0}^{\infty}$ it will be slightly more convenient to consider
the system $\big\{\psi_{k}=\frac{\varphi_{k}}{f}\big\}_{k=0}^{\infty}$ which
from the previous theorem and due to the boundedness of $\left\vert
f\right\vert $ and $1/\left\vert f\right\vert $ is also complete with respect
to the maximum norm in the space of piecewise continuously differentiable
functions defined on $[a,b]$.
\end{remark}

\section{$\ $Generalized Taylor expansions}

In the previous section we showed that the system of recursive integrals
\[
\psi_{k}(x)=\left\{
\begin{tabular}
[c]{ll}%
$X^{(k)}(x)$, & $k$ \text{odd,}\\
$\widetilde{X}^{(k)}(x)$, & $k$ \text{even}%
\end{tabular}
\ \ \ \ \ \ \ \ \ \right.
\]
is complete in $L_{2}(a,b)$ as well as\ with respect to the maximum norm in
the space of piecewise differentiable functions defined on $[a,b]$. The system
of recursive integrals generalizes the system of powers of the variable $x$.
In this section we obtain generalizations of several classical results
concerning the Taylor expansions. We start with a generalization of Taylor's formula.

Let us assume that the functions $f$ and $h$ on a certain segment $[a,b]$
possess the derivatives of all orders up to the order $n$ and that $f(x)\neq
0$, when $x\in\lbrack a,b]$. Then in $[a,b]$ the following generalized
derivatives are defined%

\[
\gamma_{0}(h)(x)=h(x),
\]

\begin{equation}
\gamma_{k}(h)(x)=\left(  f^{2}(x)\right)  ^{(-1)^{k-1}}\big(\gamma
_{k-1}(h)\big)^{\prime}(x), \label{gamma}%
\end{equation}
for $k=1,2,\ldots,n$. Let us consider a function of the form
\begin{equation}
P_{n}(x)=\sum_{k=0}^{n}\alpha_{k}\psi_{k}(x). \label{Pn}%
\end{equation}
Similarly to the fact that the coefficients of a polynomial $\Sigma_{k=0}%
^{n}a_{k}(x-x_{0})^{k}$ can be expressed through its value and the values of
its derivatives at the point $x_{0}$ we obtain that the coefficients
$\alpha_{k}$ in (\ref{Pn}) can be expressed through the value of $P_{n}$ and
the values of its generalized derivatives at the point $x_{0}$. Indeed, a
simple calculation gives us the following result
\begin{equation}
\alpha_{k}=\frac{\gamma_{k}(P_{n})(x_{0})}{k!}. \label{alphak}%
\end{equation}
Functions of the form (\ref{Pn}) will be called generalized polynomials of
order $n$. Notice that application of the generalized derivatives
(\ref{gamma}) to a generalized polynomial does not require the smoothness of
$f$,-- formula (\ref{alphak}) is true for any continuous and nonvanishing $f$.

Now let us consider a function $h$ possessing at the point $x_{0}$ the
derivatives of all orders up to the order $n$ and suppose that the same is
true for the function $f$ which additionally is different from zero at $x_{0}%
$. In relation with the function $h$ we introduce a generalized polynomial of
the form (\ref{Pn}) where the coefficients $\alpha_{k}$ are defined by the
equality%
\begin{equation}
\alpha_{k}=\frac{\gamma_{k}(h)(x_{0})}{k!}. \label{alphakH}%
\end{equation}
According to the previous observation this generalized polynomial together
with its generalized derivatives (up to the order $n$) at the point $x_{0}$
possess the same values as the function $h$ and its derivatives, $\gamma
_{k}(P_{n})(x_{0})=\gamma_{k}(h)(x_{0})$, $k=0,1,\ldots n$. We are interested
in estimating the difference between $P_{n}(x)$ and $h(x)$ for $x\neq x_{0}$.
In the next theorem we obtain a generalization of the theorem on the Taylor
remainder term in the Lagrange form.

\begin{theorem}
(Generalized Taylor theorem with the Lagrange form of the remainder
term)\label{ThGenTaylorTheorem} Let $\left\{  f,h\right\}  \subset
C^{n+1}[x_{0},b]$ and $f(x)\neq0$ there. Then for any $x\in\lbrack x_{0},b]$
there exists a number $c$ between $x_{0}$ and $x$ such that
\[
h(x)=\sum_{k=0}^{n}\frac{\gamma_{k}(h)(x_{0})}{k!}\psi_{k}(x)+\frac
{\gamma_{n+1}(h)(c)}{(n+1)!}\psi_{n+1}(x).
\]

\end{theorem}

\begin{proof}
Consider the difference $R_{n}=h-P_{n}$ and the function $\psi_{n+1}$. For
simplification of the notation let us skip the subindices, $R=R_{n}$ and
$\psi=\psi_{n+1}$. We have that $\gamma_{k}(R)(x_{0})=\gamma_{k}(\psi
)(x_{0})=0$, $k=0,1,\ldots n$. We may therefore apply Cauchy's mean value
theorem to the functions $R$ and $\psi$,
\[
\frac{R(x)}{\psi(x)}=\frac{R(x)-R(x_{0})}{\psi(x)-\psi(x_{0})}=\frac
{R^{\prime}(x_{1})}{\psi^{\prime}(x_{1})}=\frac{f^{2}(x_{1})}{f^{2}(x_{1}%
)}\frac{R^{\prime}(x_{1})}{\psi^{\prime}(x_{1})}=\frac{\gamma_{1}(R)(x_{1}%
)}{\gamma_{1}(\psi)(x_{1})}%
\]
where $x_{0}<x_{1}<x\leq b$. Another application of Cauchy's mean value
theorem gives us the equalities%
\[
\frac{\gamma_{1}(R)(x_{1})}{\gamma_{1}(\psi)(x_{1})}=\frac{\gamma_{1}%
(R)(x_{1})-\gamma_{1}(R)(x_{0})}{\gamma_{1}(\psi)(x_{1})-\gamma_{1}%
(\psi)(x_{0})}=\frac{\gamma_{1}^{\prime}(R)(x_{2})}{\gamma_{1}^{\prime}%
(\psi)(x_{2})}=\frac{f^{-2}(x_{2})}{f^{-2}(x_{2})}\frac{\gamma_{1}^{\prime
}(R)(x_{2})}{\gamma_{1}^{\prime}(\psi)(x_{2})}=\frac{\gamma_{2}(R)(x_{2}%
)}{\gamma_{2}(\psi)(x_{2})}%
\]
where $x_{0}<x_{2}<x_{1}$. Continuing this procedure we obtain
\[
\frac{\gamma_{n}(R)(x_{n})}{\gamma_{n}(\psi)(x_{n})}=\frac{\gamma_{n}%
(R)(x_{n})-\gamma_{n}(R)(x_{0})}{\gamma_{n}(\psi)(x_{n})-\gamma_{n}%
(\psi)(x_{0})}=\frac{\gamma_{n}^{\prime}(R)(x_{n+1})}{\gamma_{n}^{\prime}%
(\psi)(x_{n+1})}%
\]%
\[
=\frac{f^{(-1)^{n+1}2}(x_{n+1})}{f^{(-1)^{n+1}2}(x_{n+1})}\frac{\gamma
_{n}^{\prime}(R)(x_{n+1})}{\gamma_{n}^{\prime}(\psi)(x_{n+1})}=\frac
{\gamma_{n+1}(R)(x_{n+1})}{\gamma_{n+1}(\psi)(x_{n+1})}%
\]
where $x_{0}<x_{n+1}<x_{n}<x_{n-1}<\cdots<x\leq b$. Consequently,%
\[
\frac{R(x)}{\psi(x)}=\frac{\gamma_{n+1}(R)(x_{n+1})}{\gamma_{n+1}%
(\psi)(x_{n+1})}.
\]
Moreover, by the definition of the functions $\psi_{k}$ and of the generalized
derivatives we have that $\gamma_{n+1}(R)(x)=\gamma_{n+1}(h)(x)$ and
$\gamma_{n+1}(\psi)(x)=(n+1)!$. Then
\[
R(x)=\frac{\gamma_{n+1}(h)(c)}{(n+1)!}\psi(x),
\]
where $c=x_{n+1}$.
\end{proof}

Obviously, the classical Taylor theorem with the Lagrange form of the
remainder term is a special case of theorem \ref{ThGenTaylorTheorem} when
$f\equiv1$.

\begin{definition}
A functional series of the form
\[
\sum_{k=0}^{\infty}\frac{\gamma_{k}(h)(x_{0})}{k!}\psi_{k}(x)
\]
will be called the generalized Taylor series of the function $h$.
\end{definition}

Theorem \ref{ThGenSolSturmLiouville} gives us an important example of
uniformly convergent generalized Taylor series. Indeed, we obtain that the
quotients of the linearly independent solutions $u_{1}$, $u_{2}$ of equation
(\ref{SL}) and the particular solution $f$ of (\ref{SL0}) have the form
\[
\frac{u_{1}}{f}=\sum_{n=0}^{\infty}\frac{1+(-1)^{n}}{2n!}\lambda^{\frac{n}{2}%
}\psi_{n}\quad\text{and}\quad\frac{u_{2}}{f}=\sum_{n=0}^{\infty}%
\frac{1+(-1)^{n+1}}{2n!}\lambda^{\frac{n-1}{2}}\psi_{n}.
\]

\section{A relation between the generalized and the classical Taylor
expansions}

In this section we establish relations between the generalized Taylor
coefficients of a sufficiently smooth function and its usual Taylor coefficients.

\begin{theorem}
\label{ThRelationBetweenCoefficients}Let $\left\{  f,h\right\}  \subset
C^{n}[a,b]$ and $f(x)\neq0$, $x\in\lbrack a,b]$ and $\varphi=f^{2}$. The
following relation between the ordinary derivatives and generalized
derivatives of the function $h$ at the point $x_{0}\in\lbrack a,b]$ are valid%
\begin{equation}
\left(
\begin{tabular}
[c]{c}%
$h(x_{0})$\\
$h^{\prime}(x_{0})$\\
$h^{\prime\prime}(x_{0})$\\
$h^{\prime\prime\prime}(x_{0})$\\
\vdots\\
$h^{[n]}(x_{0})$%
\end{tabular}
\ \right)  =\left(
\begin{tabular}
[c]{cccccc}%
$1$ & $0$ & $0$ & $0$ & $\hdots$ & $0$\\
$0$ & $1/\varphi(x_{0})$ & $0$ & $0$ & $\hdots$ & $0$\\
$0$ & $a_{2,1}(x_{0})$ & $1$ & $0$ & $\hdots$ & $0$\\
$0$ & $a_{3,1}(x_{0})$ & $a_{3,2}(x_{0})$ & $1/\varphi(x_{0})$ & $\hdots$ &
$0$\\
$\vdots$ & $\vdots$ & $\vdots$ & $\vdots$ & $\ddots$ & $\vdots$\\
$0$ & $a_{n,1}(x_{0})$ & $a_{n,2}(x_{0})$ & $a_{n,3}(x_{0})$ & $\hdots$ &
$\varphi^{-\frac{1+(-1)^{n+1}}{2}}(x_{0})$%
\end{tabular}
\ \right)  \left(
\begin{tabular}
[c]{c}%
$\gamma_{0}(h)(x_{0})$\\
$\gamma_{1}(h)(x_{0})$\\
$\gamma_{2}(h)(x_{0})$\\
$\gamma_{3}(h)(x_{0})$\\
$\vdots$\\
$\gamma_{n}(h)(x_{0})$%
\end{tabular}
\ \right)  \label{relation}%
\end{equation}
where the functions $a_{n,m}$ are calculated following the recursive
procedure:%
\begin{equation}
a_{n,m}=a_{n-1,m}^{\prime}+\varphi^{(-1)^{m}}a_{n-1,m-1}. \label{rule}%
\end{equation}

\end{theorem}

\begin{proof}
The first two rows in (\ref{relation}) we obtain directly from the definition
of the generalized derivatives. To obtain $h^{\prime\prime}$ we proceed in the
following way%
\[
h^{\prime\prime}=\left(  \frac{\gamma_{1}(h)}{\varphi}\right)  ^{\prime}%
=\frac{1}{\varphi}\big(  \gamma_{1}(h)\big)  ^{\prime}+\left(  \frac
{1}{\varphi}\right)  ^{\prime}\gamma_{1}(h).
\]
From (\ref{gamma}) we have
\begin{equation}
\left(  \gamma_{k}(h)\right)  ^{\prime}=\varphi^{(-1)^{k+1}}\gamma_{k+1}(h).
\label{gammak}%
\end{equation}
In particular, $\left(  \gamma_{1}(h)\right)  ^{\prime}=\varphi\gamma_{2}(h)$.
Thus, $h^{\prime\prime}=\gamma_{2}(h)+\left(  \frac{1}{\varphi}\right)
^{\prime}\gamma_{1}(h)$ (and hence $a_{2,1}=\left(  \frac{1}{\varphi}\right)
^{\prime}$).

In general, assume that we have calculated the row corresponding to $h^{[k]}$
where $k$ is even,%
\[
h^{[k]}=\gamma_{k}(h)+a_{k,k-1}\gamma_{k-1}(h)+\cdots+a_{k,1}\gamma_{1}(h).
\]
Differentiating this equality we obtain
\[
h^{[k+1]}=\left(  \gamma_{k}(h)\right)  ^{\prime}+a_{k,k-1}\left(
\gamma_{k-1}(h)\right)  ^{\prime}+\cdots+a_{k,1}\left(  \gamma_{1}(h)\right)
^{\prime}+a_{k,k-1}^{\prime}\gamma_{k-1}(h)+\cdots+a_{k,1}^{\prime}\gamma
_{1}(h).
\]
Replacing the derivatives of $\gamma_{j}(h)$, $j=1,\ldots, k$ with
$\varphi^{(-1)^{j+1}}\gamma_{j+1}(h)$ we arrive at (\ref{rule}). For an odd
$k$ the reasoning is analogous.
\end{proof}

The transformation matrix in (\ref{relation}) will be denoted by $A_{n}$.

Meanwhile the recursive procedure (\ref{rule}) allows one to calculate the
first few rows of the transformation matrix $A_{n}$ the next statement gives
us a general formula for its elements.

\begin{proposition}
The element $a_{n,m}$ with $1\leq n\leq N$ and $2\leq m\leq n$ of the matrix
$A_{N}$, $N=2,3,\ldots$ in (\ref{relation}) has the form%
\[
a_{n,m}=\sum_{k=m-1}^{n-1}\binom{n-1}{k}\left(  \frac{1}{\varphi}\right)
^{[n-1-k]}b_{k,m-1}%
\]
where $\binom{p}{q}$ represents the binomial coefficients,
\[%
\begin{split}
b_{k,m}  &  =\sum_{k_{1}=m-1}^{k-1}\binom{k-1}{k_{1}}\varphi^{\lbrack
k-1-k_{1}]}\sum_{k_{2}=m-2}^{k_{1}-1}\binom{k_{1}-1}{k_{2}}\left(  \frac
{1}{\varphi}\right)  ^{[k_{1}-1-k_{2}]}\cdots\\
&  \sum_{k_{m-2}=2}^{k_{m-3}-1}\binom{k_{m-3}-1}{k_{m-2}}\left(  \frac
{1}{\varphi}\right)  ^{[k_{m-3}-1-k_{m-2}]}\sum_{k_{m-1}=1}^{k_{m-2}-1}%
\binom{k_{m-2}-1}{k_{m-1}}\varphi^{\lbrack k_{m-2}-1-k_{m-1}]}\left(  \frac
{1}{\varphi}\right)  ^{[k_{m-1}-1]}%
\end{split}
\]
when $m$ is even and
\[%
\begin{split}
b_{k,m}  &  =\sum_{k_{1}=m-1}^{k-1}\binom{k-1}{k_{1}}\varphi^{\lbrack
k-1-k_{1}]}\sum_{k_{2}=m-2}^{k_{1}-1}\binom{k_{1}-1}{k_{2}}\left(  \frac
{1}{\varphi}\right)  ^{[k_{1}-1-k_{2}]}\cdots\\
&  \sum_{k_{m-2}=2}^{k_{m-3}-1}\binom{k_{m-3}-1}{k_{m-2}}\varphi^{\lbrack
k_{m-3}-1-k_{m-2}]}\sum_{k_{m-1}=1}^{k_{m-2}-1}\binom{k_{m-2}-1}{k_{m-1}%
}\left(  \frac{1}{\varphi}\right)  ^{[k_{m-2}-1-k_{m-1}]}\varphi^{\lbrack
k_{m-1}-1]}%
\end{split}
\]
when $m$ is odd (in both cases, $2\leq m\leq k$).
\end{proposition}

\begin{proof}
Let us notice that
\begin{equation}
h^{\prime}=\frac{1}{\varphi}\gamma_{1} \label{hprime}%
\end{equation}
(we will write $\gamma_{k}$ understanding $\gamma_{k}(h)$) and recall the
formula for the $n$-th derivative of a product of two functions
\[
(fg)^{[n]}=\sum_{k=0}^{n}\binom{n}{k}f^{[n-k]}g^{[k]}.
\]
Applying this formula in order to evaluate the $(n-1)$-th derivative of the
equality (\ref{hprime}) we obtain
\begin{equation}
h^{[n]}=\sum_{k=0}^{n-1}\binom{n-1}{k}\left(  \frac{1}{\varphi}\right)
^{[n-1-k]}\gamma_{1}^{[k]},\quad n\geq1. \label{h_n}%
\end{equation}
Considering $\gamma_{1}^{[k]}$, $k\geq1$ we obtain
\[%
\begin{split}
\gamma_{1}^{[k]}  &  =\sum_{k_{1}=0}^{k-1}\binom{k-1}{k_{1}}\varphi^{\lbrack
k-1-k_{1}]}\gamma_{2}^{[k_{1}]}\\
&  =\varphi^{\lbrack k-1]}\gamma_{2}+\sum_{k_{1}=1}^{k-1}\binom{k-1}{k_{1}%
}\varphi^{\lbrack k-1-k_{1}]}\gamma_{2}^{[k_{1}]}\\
&  =\varphi^{\lbrack k-1]}\gamma_{2}+\sum_{k_{1}=1}^{k-1}\binom{k-1}{k_{1}%
}\varphi^{\lbrack k-1-k_{1}]}\sum_{k_{2}=0}^{k_{1}-1}\binom{k_{1}-1}{k_{2}%
}\left(  \frac{1}{\varphi}\right)  ^{[k_{1}-1-k_{2}]}\gamma_{3}^{[k_{2}]}\\
&  =\varphi^{\lbrack k-1]}\gamma_{2}+\sum_{k_{1}=1}^{k-1}\binom{k-1}{k_{1}%
}\varphi^{\lbrack k-1-k_{1}]}\left(  \frac{1}{\varphi}\right)  ^{[k_{1}%
-1]}\gamma_{3}+\\
&  +\sum_{k_{1}=2}^{k-1}\binom{k-1}{k_{1}}\varphi^{\lbrack k-1-k_{1}]}%
\sum_{k_{2}=1}^{k_{1}-1}\binom{k_{1}-1}{k_{2}}\left(  \frac{1}{\varphi
}\right)  ^{[k_{1}-1-k_{2}]}\gamma_{3}^{[k_{2}]}\\
&  =\varphi^{\lbrack k-1]}\gamma_{2}+\sum_{k_{1}=1}^{k-1}\binom{k-1}{k_{1}%
}\varphi^{\lbrack k-1-k_{1}]}\left(  \frac{1}{\varphi}\right)  ^{[k_{1}%
-1]}\gamma_{3}+\\
&  \sum_{k_{1}=2}^{k-1}\binom{k-1}{k_{1}}\varphi^{\lbrack k-1-k_{1}]}%
\sum_{k_{2}=1}^{k_{1}-1}\binom{k_{1}-1}{k_{2}}\left(  \frac{1}{\varphi
}\right)  ^{[k_{1}-1-k_{2}]}\sum_{k_{3}=0}^{k_{2}-1}\binom{k_{2}-1}{k_{3}%
}\varphi^{\lbrack k_{2}-1-k_{3}]}\gamma_{4}^{[k_{3}]}\\
&  =\varphi^{\lbrack k-1]}\gamma_{2}+\sum_{k_{1}=1}^{k-1}\binom{k-1}{k_{1}%
}\varphi^{\lbrack k-1-k_{1}]}\left(  \frac{1}{\varphi}\right)  ^{[k_{1}%
-1]}\gamma_{3}+\\
&  \sum_{k_{1}=2}^{k-1}\binom{k-1}{k_{1}}\varphi^{\lbrack k-1-k_{1}]}%
\sum_{k_{2}=1}^{k_{1}-1}\binom{k_{1}-1}{k_{2}}\left(  \frac{1}{\varphi
}\right)  ^{[k_{1}-1-k_{2}]}\varphi^{\lbrack k_{2}-1]}\gamma_{4}+\\
&  \sum_{k_{1}=3}^{k-1}\binom{k-1}{k_{1}}\varphi^{\lbrack k-1-k_{1}]}%
\sum_{k_{2}=2}^{k_{1}-1}\binom{k_{1}-1}{k_{2}}\left(  \frac{1}{\varphi
}\right)  ^{[k_{1}-1-k_{2}]}\sum_{k_{3}=1}^{k_{2}-1}\binom{k_{2}-1}{k_{3}%
}\varphi^{\lbrack k_{2}-1-k_{3}]}\gamma_{4}^{[k_{3}]}.
\end{split}
\]
Continuing this procedure we have
\[
\gamma_{1}^{[k]}=b_{k,1}\gamma_{2}+b_{k,2}\gamma_{3}+b_{k,3}\gamma_{4}%
+\cdots+b_{k,k}\gamma_{k+1}%
\]
where
\[
b_{k,1}=\varphi^{\lbrack k-1]}%
\]%
\[
b_{k,2}=\sum_{k_{1}=1}^{k-1}\binom{k-1}{k_{1}}\varphi^{\lbrack k-1-k_{1}%
]}\left(  \frac{1}{\varphi}\right)  ^{[k_{1}-1]}%
\]%
\[
b_{k,3}=\sum_{k_{1}=2}^{k-1}\binom{k-1}{k_{1}}\varphi^{\lbrack k-1-k_{1}]}%
\sum_{k_{2}=1}^{k_{1}-1}\binom{k_{1}-1}{k_{2}}\left(  \frac{1}{\varphi
}\right)  ^{[k_{1}-1-k_{2}]}\varphi^{\lbrack k_{2}-1]}%
\]%
\[
b_{k,4}=\sum_{k_{1}=3}^{k-1}\binom{k-1}{k_{1}}\varphi^{\lbrack k-1-k_{1}]}%
\sum_{k_{2}=2}^{k_{1}-1}\binom{k_{1}-1}{k_{2}}\left(  \frac{1}{\varphi
}\right)  ^{[k_{1}-1-k_{2}]}\sum_{k_{3}=1}^{k_{2}-1}\binom{k_{2}-1}{k_{3}%
}\varphi^{\lbrack k_{2}-1-k_{3}]}\left(  \frac{1}{\varphi}\right)
^{[k_{3}-1]}%
\]
and in general $b_{k,m}$ are defined by the equalities from the statement of
the proposition.

From (\ref{h_n}) we obtain for $n\geq1$
\[%
\begin{split}
h^{[n]}  &  =\sum_{k=0}^{n-1}\binom{n-1}{k}\left(  \frac{1}{\varphi}\right)
^{[n-1-k]}\gamma_{1}^{[k]}\\
&  =\left(  \frac{1}{\varphi}\right)  ^{[n-1]}\gamma_{1}+\sum_{k=1}%
^{n-1}\binom{n-1}{k}\left(  \frac{1}{\varphi}\right)  ^{[n-1-k]}\gamma
_{1}^{[k]}\\
&  =\left(  \frac{1}{\varphi}\right)  ^{[n-1]}\gamma_{1}+\sum_{k=1}%
^{n-1}\binom{n-1}{k}\left(  \frac{1}{\varphi}\right)  ^{[n-1-k]}\sum_{m=1}%
^{k}b_{k,m}\gamma_{m+1}\\
&  =a_{n,1}\gamma_{1}+a_{n,2}\gamma_{2}+a_{n,3}\gamma_{3}+a_{n,4}\gamma
_{4}+\cdots+a_{n,n-1}\gamma_{n-1}+a_{n,n}\gamma_{n},
\end{split}
\]
where
\[
a_{n,1}=\left(  \frac{1}{\varphi}\right)  ^{[n-1]}%
\]%
\[
a_{n,2}=\sum_{k=1}^{n-1}\binom{n-1}{k}\left(  \frac{1}{\varphi}\right)
^{[n-1-k]}b_{k,1}%
\]%
\[
a_{n,3}=\sum_{k=2}^{n-1}\binom{n-1}{k}\left(  \frac{1}{\varphi}\right)
^{[n-1-k]}b_{k,2}%
\]%
\[
a_{n,4}=\sum_{k=3}^{n-1}\binom{n-1}{k}\left(  \frac{1}{\varphi}\right)
^{[n-1-k]}b_{k,3}%
\]%
\[
a_{n,5}=\sum_{k=4}^{n-1}\binom{n-1}{k}\left(  \frac{1}{\varphi}\right)
^{[n-1-k]}b_{k,4}%
\]%
\[
\vdots
\]%
\[
a_{n,n-1}=\sum_{k=n-2}^{n-1}\binom{n-1}{k}\left(  \frac{1}{\varphi}\right)
^{[n-1-k]}b_{k,n-2}%
\]%
\[
a_{n,n}=\sum_{k=n-1}^{n-1}\binom{n-1}{k}\left(  \frac{1}{\varphi}\right)
^{[n-1-k]}b_{k,n-1}
\]
and hence
\[
a_{n,m}=\sum_{k=m-1}^{n-1}\binom{n-1}{k}\left(  \frac{1}{\varphi}\right)
^{[n-1-k]}b_{k,m-1},\quad2\leq m\leq n.
\]

\end{proof}

\section{Taylor coefficients of solutions of the Sturm-Liouville equation}

In the present section we obtain an interesting corollary of theorem
\ref{ThRelationBetweenCoefficients}. Given a nonvanishing (complex valued)
solution of equation (\ref{SLhom}) possessing $n$ derivatives at a certain
point, for any solution of (\ref{SL}) (for any value of the spectral parameter
$\lambda$) we are able to calculate the exact values of its $n$ derivatives at
the same point without any integration. Moreover, each derivative\ is a
polynomial of a certain order with respect to $\lambda$.

\begin{theorem}
\label{ThTaylor_foru1u2}Let $(a,b)$ be a finite interval and $f\in
C^{2}(a,b)\cap C^{1}[a,b]$ be a complex valued solution of (\ref{SLhom}) such
that $f(x)\neq0$ for any $x\in\lbrack a,b]$ and at the point $x_{0}\in\lbrack
a,b]$ there exist the derivatives of $f$ up to the $n$-th order. Then the
linearly independent solutions $u_{1}$ and $u_{2}$ of equation (\ref{SL})
satisfying the initial conditions (\ref{initial1}) and (\ref{initial2})
respectively possess at the point $x_{0}$ the derivatives up to the $n$-th
order which can be calculated according to the following relations%
\begin{equation}
\left(
\begin{tabular}
[c]{c}%
$u_{1}(x_{0})/f(x_{0})$\\
$\left(  u_{1}/f\right)  ^{\prime}(x_{0})$\\
$\left(  u_{1}/f\right)  ^{\prime\prime}(x_{0})$\\
$\left(  u_{1}/f\right)  ^{\prime\prime\prime}(x_{0})$\\
$\left(  u_{1}/f\right)  ^{IV}(x_{0})$\\
$\vdots$\\
$\left(  u_{1}/f\right)  ^{[n]}(x_{0})$%
\end{tabular}
\ \ \ \right)  =\quad A_{n}\quad\left(
\begin{tabular}
[c]{c}%
$1$\\
$0$\\
$\lambda$\\
$0$\\
$\lambda^{2}$\\
$\vdots$\\
$\frac{1+(-1)^{n}}{2}\lambda^{\frac{n}{2}}$%
\end{tabular}
\ \ \ \right)  \label{for_u1}%
\end{equation}
and
\[
\left(
\begin{tabular}
[c]{c}%
$u_{2}(x_{0})/f(x_{0})$\\
$\left(  u_{2}/f\right)  ^{\prime}(x_{0})$\\
$\left(  u_{2}/f\right)  ^{\prime\prime}(x_{0})$\\
$\left(  u_{2}/f\right)  ^{\prime\prime\prime}(x_{0})$\\
$\left(  u_{2}/f\right)  ^{IV}(x_{0})$\\
$\vdots$\\
$\left(  u_{2}/f\right)  ^{[n]}(x_{0})$%
\end{tabular}
\ \ \ \right)  =\quad A_{n}\quad\left(
\begin{tabular}
[c]{c}%
$0$\\
$1$\\
$0$\\
$\lambda$\\
$0$\\
$\vdots$\\
$\frac{1+(-1)^{n+1}}{2}\lambda^{\frac{n-1}{2}}$%
\end{tabular}
\ \ \ \right)
\]
where $A_{n}$ is the matrix from (\ref{relation}).
\end{theorem}

\begin{proof}
From theorem \ref{ThGenSolSturmLiouville} we have that in $(a,b)$ the
functions $u_{1}$ and $u_{2}$ defined by (\ref{u1u2}) are solutions of the
Sturm-Liouville equation (\ref{SL}) satisfying the initial conditions
(\ref{initial1}) and (\ref{initial2}), and the series in (\ref{u1u2}) converge
uniformly on $[a,b]$. Thus, the functions $u_{1}/f$ and $u_{2}/f$ are expended
into uniformly convergent generalized Taylor series, and from (\ref{u1u2}) we
have that
\[
\gamma_{k}(\frac{u_{1}}{f})(x_{0})=\frac{1+(-1)^{k}}{2}\lambda^{\frac{k}{2}}%
\]
and
\[
\gamma_{k}(\frac{u_{2}}{f})(x_{0})=\frac{1+(-1)^{k+1}}{2}\lambda^{\frac
{k-1}{2}}.
\]

\end{proof}

\begin{example}
Our first example is the equation (\ref{SL}) with a constant coefficient
$q=-c^{2}$, $c\in\mathbb{R}$, $c\neq0$. Obviously, a nonvanishing solution of
(\ref{SLhom}) can be chosen in the form $f(x)=e^{cx}$. Considering $x_{0}=0$
we calculate the first six rows of the matrix $A_{n}$. We have%
\[
A_{5}=\left(
\begin{tabular}
[c]{cccccc}%
$1$ & $0$ & $0$ & $0$ & $0$ & $0$\\
$0$ & $1$ & $0$ & $0$ & $0$ & $0$\\
$0$ & $-2c$ & $1$ & $0$ & $0$ & $0$\\
$0$ & $4c^{2}$ & $-2c$ & $1$ & $0$ & $0$\\
$0$ & $-8c^{3}$ & $4c^{2}$ & $-4c$ & $1$ & $0$\\
$0$ & $16c^{4}$ & $-8c^{3}$ & $12c^{2}$ & $-4c$ & $1$%
\end{tabular}
\ \ \right)  .
\]
On the other hand it is easy to see that the solution $u_{1}$ from
(\ref{u1u2}) in this case has the form
\[
u_{1}(x)=\frac{c+\kappa}{2\kappa}e^{\kappa x}+\frac{\kappa-c}{2\kappa
}e^{-\kappa x},
\]
where $\kappa=\sqrt{c^{2}+\lambda}$. Simple calculation involving the quotient
$u_{1}/f=\frac{c+\kappa}{2\kappa}e^{\left(  \kappa-c\right)  x}+\frac
{\kappa-c}{2\kappa}e^{-\left(  \kappa+c\right)  x}$ gives us the following
values
\[
\left(  u_{1}/f\right)  ^{\prime}(0)=0,\quad\left(  u_{1}/f\right)
^{\prime\prime}(0)=\lambda,\quad\left(  u_{1}/f\right)  ^{\prime\prime\prime
}(0)=-2c\lambda,
\]%
\[
\left(  u_{1}/f\right)  ^{IV}(0)=4c^{2}\lambda+\lambda^{2},\quad\left(
u_{1}/f\right)  ^{V}(0)=-8c^{3}\lambda-4c\lambda^{2}%
\]
which is exactly the result of multiplication of the matrix $A_{5}$ by the
vector
\[
(1,\ 0,\ \lambda,\ 0,\ \lambda^{2},\ 0)^{\top}%
\]
(see (\ref{for_u1})).
\end{example}

\begin{example}
Consider the equation
\begin{equation}
u^{\prime\prime}(x)-a^{2}x^{2k}u(x)=\lambda u(x). \label{ex2lambda}%
\end{equation}
This equation with $\lambda=-1$ can be found, e.g., in \cite[Part 3, Chapter
II, eq. 2.15]{Kamke} where it is stated that its solution is unknown,
meanwhile the solution of (\ref{ex2lambda}) when $\lambda=0$ is known
\cite[Part 3, Chapter II, eq. 2.14]{Kamke} even for noninteger $k$. In
particular, when $k=-\frac{2n}{2n+1}$ and $n$ is a negative integer number the
general solution of the equation
\begin{equation}
y^{\prime\prime}(x)-a^{2}x^{2k}y(x)=0 \label{ex2hom}%
\end{equation}
has the form
\[
y(x)=\left(  x^{\frac{2n-1}{2n+1}}\frac{d}{dx}\right)  ^{-n}\left(  c_{1}%
\exp\left(  a(2n+1)x^{\frac{1}{2n+1}}\right)  +c_{2}\exp\left(
-a(2n+1)x^{\frac{1}{2n+1}}\right)  \right)  .
\]
This solution can be used for solving (\ref{ex2lambda}). Consider, for
example, $k=-2$ ($n=-1$). The general solution of (\ref{ex2hom}) takes the
form
\[
y(x)=ax(c_{1}e^{-a/x}-c_{2}e^{a/x}).
\]
As a particular solution $f$ we can take, e.g., $f(x)=axe^{a/x}$. Fix
$x_{0}=1$. Calculation of the first five rows of the matrix $A_{n}$ gives us
the following result
\[
A_{4}=\left(
\begin{tabular}
[c]{ccccc}%
$1$ & $0$ & $0$ & $0$ & $0$\\
$0$ & $\frac{e^{-2a}}{a^{2}}$ & $0$ & $0$ & $0$\\
$0$ & $\frac{2(a-1)e^{-2a}}{a^{2}}$ & $1$ & $0$ & $0$\\
$0$ & $e^{-2a}\left(  4+\frac{6}{a^{2}}(1-2a)\right)  $ & $2(a-1)$ &
$\frac{e^{-2a}}{a^{2}}$ & $0$\\
$0$ & $e^{-2a}\left(  24-16a-\frac{24(1-a)}{a^{2}}\right)  $ & $8-16a+4a^{2}$
& $\frac{4(a-1)}{a^{2}}e^{-2a}$ & $1$%
\end{tabular}
\ \right)
\]
from where using theorem \ref{ThTaylor_foru1u2} the first four derivatives of
the functions $u_{1}/f$ and $u_{2}/f$ at $x_{0}=1$ can be calculated.
\end{example}

\section{Conclusions}

A formula for calculating Taylor coefficients of the solution of the
Sturm-Liouville equation is obtained. It is based on the knowledge of the
Taylor coefficients of a particular solution of the equation corresponding to
the zero value of a spectral parameter. The form of the (matrix) map
transforming one set of Taylor coefficients into the other is independent of
the particular form of the particular solution.

\section*{Acknowledgements}

Research of V.K. was supported by CONACYT, Mexico via the research project
50424. S.T. wants to thank the Department of Mathematics of the Cinvestav,
campus Queretaro for the hospitality during a part of his sabbatical year. The
research of S.T. is partly supported by grant from NSERC of Canada.

\end{document}